\documentclass[12pt]{article}

\usepackage[margin=1in]{geometry}
\usepackage{enumerate}
\usepackage[hidelinks]{hyperref}
\usepackage{amssymb}
\usepackage{mathrsfs}
\usepackage{graphicx}
\usepackage{amsthm}
\usepackage{amsmath}
\usepackage{bbm}

\newtheorem{definition}{Definition}[section]
\newtheorem{theorem}[definition]{Theorem}
\newtheorem{lemma}[definition]{Lemma}
\newtheorem{corollary}[definition]{Corollary}

\newtheorem{example}[definition]{Example}

\newtheorem{proposition}[definition]{Proposition}

\def\makeCal#1{%
\expandafter\newcommand\csname c#1\endcsname{\mathcal{#1}}}
\def\makeBf#1{%
\expandafter\newcommand\csname b#1\endcsname{\mathbf{#1}}}
\def\makeBb#1{%
\expandafter\newcommand\csname m#1\endcsname{\mathbb{#1}}}
\def\makeFrak#1{%
\expandafter\newcommand\csname f#1\endcsname{\mathfrak{#1}}}
\def\makeScr#1{%
\expandafter\newcommand\csname s#1\endcsname{\mathscr{#1}}}

\makeatletter
\count@=0
\loop
\advance\count@ 1
\edef\y{\@Alph\count@}%
\expandafter\makeCal\y
\expandafter\makeBf\y
\expandafter\makeBb\y
\expandafter\makeFrak\y
\expandafter\makeScr\y
\ifnum\count@<26
\repeat
\makeatother

\DeclareMathOperator{\m1}{\mathbbm{1}}
\newcommand{\tG}{\tilde{G}}

\usepackage{xypic}

\title{Doubly alternating words in the positive part of $U_q(\widehat{\mathfrak{sl}}_2)$}
\author{Chenwei Ruan}
\date{}

\begin{document}
\maketitle

\begin{abstract}
This paper is about the positive part $U_q^+$ of the $q$-deformed enveloping algebra $U_q(\widehat{\mathfrak{sl}}_2)$. The algebra $U_q^+$ admits an embedding, due to Rosso, into a $q$-shuffle algebra $\mV$. The underlying vector space of $\mV$ is the free algebra on two generators $x,y$. Therefore, the algebra $\mV$ has a basis consisting of the words in $x,y$. Let $U$ denote the image of $U_q^+$ under the Rosso embedding. In our first main result, we find all the words in $x,y$ that are contained in $U$. One type of solution is called alternating. The alternating words have been studied by Terwilliger. There is another type of solution, which we call doubly alternating. In our second main result, we display many commutator relations involving the doubly alternating words. In our third main result, we describe how the doubly alternating words are related to the alternating words. 

\end{abstract}

\section{Introduction}
The $q$-deformed enveloping algebra $U_q(\widehat{\mathfrak{sl}}_2)$ has a subalgebra $U_q^+$, called the positive part \cite{CP,lusztig}. The algebras $U_q(\widehat{\mathfrak{sl}}_2)$ and $U_q^+$ appear in the topics of combinatorics \cite{ariki,IT2,JKKKY,watanabe}, mathematical physics \cite{baseilhac1,BS,DFJMN,jimbo}, and representation theory \cite{BCP,bittmann,FR,jing}. Both $U_q(\widehat{\mathfrak{sl}}_2)$ and $U_q^+$ are associative, noncommutative, and infinite-dimensional. 

\medskip
\noindent The algebra $U_q^+$ is defined by two generators $A,B$ and two $q$-Serre relations; see \eqref{eq:qSerre1}, \eqref{eq:qSerre2} below. 

\medskip
\noindent In \cite{rosso1,rosso2}, Rosso obtained an embedding of $U_q^+$ into a $q$-shuffle algebra $\mV$. The Rosso embedding has been used to study $U_q^+$; see for example \cite{leclerc1,inverse,Delta_n,ter_alternating,ter_catalan,ter_beck}. The underlying vector space of $\mV$ is the free algebra on two generators $x,y$. Consequently, the algebra $\mV$ has a basis consisting of the words in $x,y$. By construction, the Rosso embedding sends $A \mapsto x$ and $B \mapsto y$. Let $U$ denote the subalgebra of the $q$-shuffle algebra $\mV$ generated by $x,y$. Note that $U$ is the image of $U_q^+$ under the Rosso embedding. 

\medskip
\noindent In this paper, our first main result is to classify all the words in $x,y$ that are contained in $U$. As we will see, these words fall into one of the three following types: 
\begin{enumerate}
\item $\cdots xxx \cdots$ or $\cdots yyy \cdots$
\item $\cdots xyxyxy \cdots$
\item $\cdots xxyyxxyyxxyy \cdots$
\end{enumerate}

\noindent A word of type (i) is generated by $x$ or $y$ with respect to the $q$-shuffle product. 

\medskip
\noindent A word of type (ii) is called alternating. The alternating words were closely examined in \cite{ter_alternating}. In \cite[Theorems 10.1 and 10.2]{ter_alternating}, the alternating words are used to construct PBW bases for $U_q^+$. 

\medskip
\noindent A word of type (iii) is said to be doubly alternating. In this paper, we will study the doubly alternating words. As part of this study, we will express the doubly alternating words in terms of the alternating words. 

\medskip
\noindent This paper is organized as follows. In Section 2, we recall the algebra $U_q^+$ and its Rosso embedding. In Section 3, we classify the words contained in $U$. In Section 4, we give a commutator relation for each choice of one letter and one doubly alternating word. In Section 5, we express each doubly alternating word as a polynomial in the alternating words. In Section 6, we express the results in Section 5 in terms of generating functions. 

\section{The algebra $U_q^+$ and its Rosso embedding}
Before starting our formal argument, we first establish some conventions and notations. Recall the natural numbers $\mN=\{0,1,2,\ldots\}$ and the integers $\mZ=\{0,\pm 1,\pm 2,\ldots\}$. Let $\mF$ denote a field of characteristic zero. All algebras discussed in this paper are associative, over $\mF$, and have a multiplicative identity. Let $q$ denote a nonzero scalar in $\mF$ that is not a root of unity. For elements $X,Y$ in any algebra, their commutator and $q$-commutator are given by 
\begin{align*}
[X,Y]=XY-YX, \hspace{4em} [X,Y]_q=qXY-q^{-1}YX. 
\end{align*}

\noindent For $n \in \mZ$, define 
\begin{equation}\label{eq:qint}
[n]_q=\frac{q^n-q^{-n}}{q-q^{-1}}. 
\end{equation}

\noindent Let $U_q^+$ denote the algebra with generators $A,B$ and relations 
\begin{equation}\label{eq:qSerre1}
A^3B-[3]_qA^2BA+[3]_qABA^2-BA^3=0, 
\end{equation}
\begin{equation}\label{eq:qSerre2}
B^3A-[3]_qB^2AB+[3]_qBAB^2-AB^3=0. 
\end{equation}

\noindent The relations \eqref{eq:qSerre1}, \eqref{eq:qSerre2} are the $q$-Serre relations. 

\medskip
\noindent We now recall the $q$-shuffle algebra $\mV$. Let $x,y$ denote noncommuting indeterminates. We call $x$ and $y$ \textit{letters}. The underlying vector space of $\mV$ is the free algebra generated by the letters $x,y$. For $n \in \mN$, the product of $n$ letters is called a \textit{word} of \textit{length} $n$. The word of length $0$ is called \textit{trivial} and denoted by $\m1$. The words form a basis for the vector space $\mV$; this basis is called \textit{standard}. 

\medskip
\noindent The vector space $\mV$ can be equipped with another algebra structure, called the $q$-shuffle algebra \cite{rosso1,rosso2}. The $q$-shuffle product is denoted by $\star$. The following recursive definition of $\star$ is adopted from \cite{green}. 

\begin{itemize}
\item For $v \in \mV$, 
	\begin{equation*}\label{star1}
	\m1 \star v=v \star \m1=v. 
	\end{equation*}
\item For the letters $u,v$, 
	\begin{equation*}\label{star2}
	u \star v=uv+vuq^{\langle u,v \rangle}, 
	\end{equation*}	
	where 
	 \begin{equation*}
	\langle x,x \rangle=\langle y,y \rangle =2, \hspace{4em}\langle x,y \rangle=\langle y,x \rangle=-2.
	\end{equation*}
\item For a letter $u$ and a nontrivial word $v=v_1v_2 \cdots v_n$ in $\mV$, 
	\begin{equation*}\label{star3.1}
	u \star v=\sum_{i=0}^n v_1 \cdots v_iuv_{i+1} \cdots v_n q^{\langle u,v_1 \rangle+\cdots+\langle u,v_i \rangle}, 
	\end{equation*}
	\begin{equation*}\label{star3.2}
	v \star u=\sum_{i=0}^n v_1 \cdots v_iuv_{i+1} \cdots v_n q^{\langle u,v_n \rangle+\cdots+\langle u,v_{i+1} \rangle}. 
	\end{equation*}
\item For nontrivial words $u=u_1u_2 \cdots u_r$ and $v=v_1v_2 \cdots v_s$ in $\mV$, 
	\begin{equation*}\label{star4.1}
	u \star v=u_1((u_2 \cdots u_r) \star v)+v_1(u \star (v_2 \cdots v_s))q^{\langle v_1,u_1 \rangle+\cdots+\langle v_1,u_r \rangle}, 
	\end{equation*}
	\begin{equation*}\label{star4.2}
	u \star v=(u \star (v_1 \cdots v_{s-1}))v_s+((u_1 \cdots u_{r-1}) \star v)u_rq^{\langle u_r,v_1 \rangle+\cdots+\langle u_r,v_s \rangle}. 
	\end{equation*}
\end{itemize}

\noindent It was shown in \cite{rosso1,rosso2} that the vector space $\mV$, equipped with the $q$-shuffle product $\star$, forms an algebra. This is the $q$-shuffle algebra. 

\medskip
\noindent Next we recall the Rosso embedding of $U_q^+$ into the $q$-shuffle algebra $\mV$. In \cite{rosso1,rosso2}, Rosso showed that $x,y$ satisfy 
\begin{equation*}
x \star x \star x \star y-[3]_qx \star x \star y \star x+[3]_qx \star y \star x \star x-y \star x \star x \star x=0, 
\end{equation*}
\begin{equation*}
y \star y \star y \star x-[3]_qy \star y \star x \star y+[3]_qy \star x \star y \star y-x \star y \star y \star y=0. 
\end{equation*}

\noindent As a result, there exists an algebra homomorphism $\natural$ from $U_q^+$ to the $q$-shuffle algebra $\mV$ that sends $A \mapsto x$ and $B \mapsto y$. By \cite[Theorem 15]{rosso2} the map $\natural$ is injective. Let $U$ denote the subalgebra of $\mV$ generated by $x,y$ with respect to the $q$-shuffle product. Then the map $\natural:U_q^+ \to U$ is an algebra isomorphism. 

\medskip
\noindent Recall that the words form the standard basis for $\mV$. This naturally raises a question: which words are contained in $U$? We will answer this question in Section 3. 

\section{Classification of the words in $U$}
In this section, we classify the words that are contained in $U$. We start by recalling a few definitions and results. 

\begin{definition}\label{def:biform}\rm
(See \cite[Definition 4.1]{PT}.) Let $(~,~):\mV \times \mV \to \mF$ denote the bilinear form with respect to which the standard basis for $\mV$ is orthonormal. 
\end{definition}

\begin{definition}\rm\label{def:idealJ}
(See \cite[Definition 3.5]{PT}.) Let $J$ denote the two-sided ideal of the free algebra $\mV$ generated by 
\begin{equation*}
xxxy-[3]_qxxyx+[3]_qxyxx-yxxx, \hspace{3em}yyyx-[3]_qyyxy+[3]_qyxyy-xyyy. 
\end{equation*}
\end{definition}

\begin{proposition}\rm\label{prop:orthJ}
(See \cite[Lemma 6.5]{PT}.) The algebra $U$ is the orthogonal complement of $J$ in $\mV$ with respect to the bilinear form $(~,~)$. 
\end{proposition}

\begin{lemma}\rm\label{lem:wordsinU}
Let $w$ denote a word in $\mV$. Then $w \in U$ if and only if $w$ does not contain any of the following segments 
\begin{equation*}
xxxy, \hspace{1em} xxyx, \hspace{1em} xyxx, \hspace{1em} yxxx, \hspace{1em} yyyx, \hspace{1em} yyxy, \hspace{1em} yxyy, \hspace{1em} xyyy. 
\end{equation*}
\end{lemma}
\begin{proof}
Follows from Proposition \ref{prop:orthJ}. 
\end{proof}

\noindent Using Lemma \ref{lem:wordsinU}, we obtain the following examples. 

\begin{example}\rm\label{ex:single}
For $n \in \mN$, the following words are contained in $U$. 
\begin{equation}\label{eq:single}
x^n, \hspace{3em}y^n. 
\end{equation}
\end{example}

\begin{example}\rm\label{ex:alternating}
For $n \in \mN$, the following words are contained in $U$. 
\begin{equation}\label{eq:alternating}
(xy)^{n+1}, \hspace{3em}(yx)^{n+1}, \hspace{3em}(xy)^nx, \hspace{3em}(yx)^ny. 
\end{equation}
\end{example}

\begin{example}\rm\label{ex:dalternating}
For $n \in \mN$, the following words are contained in $U$. 
\begin{align}
\label{eq:dalternating1}
&(xxyy)^{n+1}, &&(yyxx)^{n+1}, &&(xxyy)^nxx, &&(yyxx)^nyy; \\[0.2em]
\label{eq:dalternating2}
&xyy(xxyy)^n, &&yxx(yyxx)^n, &&x(yyxx)^n, &&y(xxyy)^n; \\[0.2em]
\label{eq:dalternating3}
&(xxyy)^nxxy, &&(yyxx)^nyyx, &&(xxyy)^nx, &&(yyxx)^ny; \\[0.2em]
\label{eq:dalternating4}
&x(yyxx)^ny, &&y(xxyy)^nx, &&xyy(xxyy)^nx, &&yxx(yyxx)^ny. 
\end{align}
\end{example}

\noindent We remark that the words appearing in Example \ref{ex:alternating} are called \textit{alternating}. The alternating words were defined and closely examined in \cite{ter_alternating}. For a list of known relations on the alternating words, see Appendix A. 

\medskip
\noindent The following notation for the alternating words will be used later. 

\begin{definition}\rm\label{def:tGn}
(See \cite[Definition 5.2]{ter_alternating}.) For $n \in \mN$, define
\begin{equation*}
\tG_{n+1}=(xy)^{n+1}, \hspace{2em}G_{n+1}=(yx)^{n+1}, \hspace{2em}W_{-n}=(xy)^nx, \hspace{2em}W_{n+1}=(yx)^ny. 
\end{equation*}
The above exponents are with respect to the free product. 
\end{definition}

\begin{example}\rm\label{ex:tGn}
We list $\tG_{n+1}$, $G_{n+1}$, $W_{-n}$, $W_{n+1}$ for $0 \leq n \leq 3$. 
\begin{align*}
&\tG_1=xy, && \tG_2=xyxy, && \tG_3=xyxyxy, && \tG_4=xyxyxyxy; \\
&G_1=yx, && G_2=yxyx, && G_3=yxyxyx, && G_4=yxyxyxyx; \\
&W_0=x, && W_{-1}=xyx, && W_{-2}=xyxyx, && W_{-3}=xyxyxyx; \\
&W_1=y, && W_2=yxy, && W_3=yxyxy, && W_4=yxyxyxy. 
\end{align*}
\end{example}

\noindent Motivated by the alternating words and Example \ref{ex:dalternating}, we make the following definition. 

\begin{definition}\rm\label{def:dalternating}
A word of the form 
\begin{equation*}
\cdots xxyyxxyyxxyy \cdots
\end{equation*}
is said to be \textit{doubly alternating}. 

\medskip
\noindent There are $16$ families of doubly alternating words, depending on the choice of the first two letters and last two letters. These families are listed in \eqref{eq:dalternating1}--\eqref{eq:dalternating4}. 
\end{definition}

\begin{theorem}\rm\label{thm:wordsinU}
A word is contained in $U$ if and only if it appears in Examples \ref{ex:single}, \ref{ex:alternating}, \ref{ex:dalternating}. 
\end{theorem}
\begin{proof}
Let $w$ denote a word contained in $U$. We will show that $w$ is listed in Example \ref{ex:single} or \ref{ex:alternating} or \ref{ex:dalternating}. 

\medskip
\noindent We may assume that $w$ has length at least $2$; otherwise the result is trivial. 

\medskip
\noindent If $w$ contains only one of the letters $x$, $y$, then $w$ is listed in (1). For the rest of this proof, we assume that $w$ contains both the letters $x$, $y$. By Lemma \ref{lem:wordsinU}, this assumption implies that $w$ does not contain either of the segments $xxx$, $yyy$. 

\medskip
\noindent If $w$ contains one of the segments $xyx$, $yxy$, then $w$ is listed in \eqref{eq:alternating} by Lemma \ref{lem:wordsinU}. For the rest of this proof, we assume that $w$ does not contain either of the segments $xyx$, $yxy$. 

\medskip
\noindent We consider the first and last two letters in $w$. There are $2^4=16$ choices of the four letters. Using the above assumptions and Lemma \ref{lem:wordsinU}, one can routinely check that these $16$ choices are in one-to-one correspondence with the $16$ types of words listed in \eqref{eq:dalternating1}--\eqref{eq:dalternating4}. 
\end{proof}

\section{Commutator relations involving $x,y$}
In this section, we give a commutator relation for each choice of one letter and one doubly alternating word. This commutator relation is with respect to the $q$-shuffle product. For the commutator relations involving two doubly alternating words, see Appendix B. 

\medskip
\noindent We remark that the doubly alternating words from each line \eqref{eq:dalternating1}--\eqref{eq:dalternating4} appear in separate propositions below. 

\begin{proposition}\rm\label{prop:xcommutators1}
Let $n \in \mN$. The following relations hold in $U$. 

\begin{equation}\label{eq:xcommutators1.1}
[(xxyy)^n,x]_{q^2}=(q^2-q^{-2})(xxyy)^nx, 
\end{equation}
\begin{equation}\label{eq:xcommutators1.2}
[x,(yyxx)^n]_{q^2}=(q^2-q^{-2})x(yyxx)^n, 
\end{equation}
\begin{equation}\label{eq:xcommutators1.3}
[x,(xxyy)^nxx]=0, 
\end{equation}
\begin{equation}\label{eq:xcommutators1.4}
[(yyxx)^nyy,x]=(1-q^{-4})\left((yyxx)^nyyx-xyy(xxyy)^n\right). 
\end{equation}
\end{proposition}
\begin{proof}
We first show \eqref{eq:xcommutators1.1}. 

\medskip
\noindent Clearly \eqref{eq:xcommutators1.1} holds for $n=0$. We now assume $n \geq 1$. 

\medskip
\noindent Note that 
\begin{equation}\label{eq:xcommutatorproof0}
[(xxyy)^n,x]_{q^2}=q^2(xxyy)^n \star x-q^{-2}x \star (xxyy)^n. 
\end{equation}

\medskip
\noindent We write $(xxyy)^n \star x$ (resp.\ $x \star (xxyy)^n$) as a linear combination of words. By the definition of the $q$-shuffle product from Section 2, each word appearing in the linear combination can be expressed in one of the following forms: 

\begin{equation}\label{eq:xcommutatorproof1}
(xxyy)^ixxxyy(xxyy)^j, \text{ where } i+j=n-1; 
\end{equation}
\begin{equation}\label{eq:xcommutatorproof2}
(xxyy)^ixxyxy(xxyy)^j, \text{ where } i+j=n-1; 
\end{equation}
\begin{equation}\label{eq:xcommutatorproof3}
(xxyy)^{n-1}xxyyx. 
\end{equation}

\noindent We compute the coefficients of these words in the linear combination. The coefficient is $q^{-4}+q^{-2}+1$ (resp.\ $1+q^2+q^4$) for a word from \eqref{eq:xcommutatorproof1}, the coefficient is $q^{-2}$ (resp.\ $q^2$) for a word from \eqref{eq:xcommutatorproof2}, and the coefficient is $1$ for the word $(xxyy)^{n-1}xxyyx$. We have proved \eqref{eq:xcommutators1.1}. 

\medskip
\noindent The remaining relations can be proved in a similar way. 
\end{proof}

\begin{proposition}\rm\label{prop:ycommutators1}
Let $n \in \mN$. The following relations hold in $U$. 

\begin{equation}\label{eq:ycommutators1.1}
[y,(xxyy)^n]_{q^2}=(q^2-q^{-2})y(xxyy)^n, 
\end{equation}
\begin{equation}\label{eq:ycommutators1.2}
[(yyxx)^n,y]_{q^2}=(q^2-q^{-2})(yyxx)^ny, 
\end{equation}
\begin{equation}\label{eq:ycommutators1.3}
[y,(xxyy)^nxx]=(1-q^{-4})\left(yxx(yyxx)^n-(xxyy)^nxxy\right), 
\end{equation}
\begin{equation}\label{eq:ycommutators1.4}
[(yyxx)^nyy,y]=0. 
\end{equation}
\end{proposition}
\begin{proof}
Similar to the proof of Proposition \ref{prop:xcommutators1}. 
\end{proof}

\begin{proposition}\rm\label{prop:xcommutators2}
Let $n \in \mN$. The following relations hold in $U$. 

\begin{equation}\label{eq:xcommutators2.1}
[xyy(xxyy)^n,x]_q=(q-q^{-3})\left(xyy(xxyy)^nx-(xxyy)^{n+1}\right), 
\end{equation}
\begin{equation}\label{eq:xcommutators2.2}
[x,yxx(yyxx)^n]_q=0, 
\end{equation}
\begin{equation}\label{eq:xcommutators2.3}
[x,x(yyxx)^n]_q=(q^3-q^{-1})(xxyy)^nxx, 
\end{equation}
\begin{equation}\label{eq:xcommutators2.4}
[y(xxyy)^n,x]_q=(q-q^{-3})y(xxyy)^nx. 
\end{equation}
\end{proposition}
\begin{proof}
Similar to the proof of Proposition \ref{prop:xcommutators1}. 
\end{proof}

\begin{proposition}\rm\label{prop:ycommutators2}
Let $n \in \mN$. The following relations hold in $U$. 

\begin{equation}\label{eq:ycommutators2.1}
[xyy(xxyy)^n,y]_q=0, 
\end{equation}
\begin{equation}\label{eq:ycommutators2.2}
[yxx(yyxx)^n,y]_q=(q-q^{-3})\left(yxx(yyxx)^ny-(yyxx)^{n+1}\right), 
\end{equation}
\begin{equation}\label{eq:ycommutators2.3}
[x(yyxx)^n,y]_q=(q-q^{-3})x(yyxx)^ny, 
\end{equation}
\begin{equation}\label{eq:ycommutators2.4}
[y,y(xxyy)^n]_q=(q^3-q^{-1})(yyxx)^nyy. 
\end{equation}
\end{proposition}
\begin{proof}
Similar to the proof of Proposition \ref{prop:xcommutators1}. 
\end{proof}

\begin{proposition}\rm\label{prop:xcommutators3}
Let $n \in \mN$. The following relations hold in $U$. 

\begin{equation}\label{eq:xcommutators3.1}
[(xxyy)^nxxy,x]_q=0, 
\end{equation}
\begin{equation}\label{eq:xcommutators3.2}
[x,(yyxx)^nyyx]_q=(q-q^{-3})\left(xyy(xxyy)^nx-(yyxx)^{n+1}\right), 
\end{equation}
\begin{equation}\label{eq:xcommutators3.3}
[(xxyy)^nx,x]_q=(q^3-q^{-1})(xxyy)^nxx, 
\end{equation}
\begin{equation}\label{eq:xcommutators3.4}
[x,(yyxx)^ny]_q=(q-q^{-3})x(yyxx)^ny. 
\end{equation}
\end{proposition}
\begin{proof}
Similar to the proof of Proposition \ref{prop:xcommutators1}. 
\end{proof}

\begin{proposition}\rm\label{prop:ycommutators3}
Let $n \in \mN$. The following relations hold in $U$. 

\begin{equation}\label{eq:ycommutators3.1}
[y,(xxyy)^nxxy]_q=(q-q^{-3})\left(yxx(yyxx)^ny-(xxyy)^{n+1}\right), 
\end{equation}
\begin{equation}\label{eq:ycommutators3.2}
[(yyxx)^nyyx,y]_q=0, 
\end{equation}
\begin{equation}\label{eq:ycommutators3.3}
[y,(xxyy)^nx]_q=(q-q^{-3})y(xxyy)^nx, 
\end{equation}
\begin{equation}\label{eq:ycommutators3.4}
[(yyxx)^ny,y]_q=(q^3-q^{-1})(yyxx)^nyy. 
\end{equation}
\end{proposition}
\begin{proof}
Similar to the proof of Proposition \ref{prop:xcommutators1}. 
\end{proof}

\begin{proposition}\rm\label{prop:xcommutators4}
Let $n \in \mN$. The following relations hold in $U$. 

\begin{equation}\label{eq:xcommutators4.1}
[x,x(yyxx)^ny]=(q^2-q^{-2})(xxyy)^nxxy, 
\end{equation}
\begin{equation}\label{eq:xcommutators4.2}
[y(xxyy)^nx,x]=(q^2-q^{-2})yxx(yyxx)^n, 
\end{equation}
\begin{equation}\label{eq:xcommutators4.3}
[x,xyy(xxyy)^nx]=(q^2-q^{-2})\left((xxyy)^{n+1}x-x(yyxx)^{n+1}\right), 
\end{equation}
\begin{equation}\label{eq:xcommutators4.4}
[x,yxx(yyxx)^ny]=0. 
\end{equation}
\end{proposition}
\begin{proof}
Similar to the proof of Proposition \ref{prop:xcommutators1}. 
\end{proof}

\begin{proposition}\rm\label{prop:ycommutators4}
Let $n \in \mN$. The following relations hold in $U$. 

\begin{equation}\label{eq:ycommutators4.1}
[x(yyxx)^ny,y]=(q^2-q^{-2})xyy(xxyy)^n, 
\end{equation}
\begin{equation}\label{eq:ycommutators4.2}
[y,y(xxyy)^nx]=(q^2-q^{-2})(yyxx)^nyyx, 
\end{equation}
\begin{equation}\label{eq:ycommutators4.3}
[xyy(xxyy)^nx,y]=0, 
\end{equation}
\begin{equation}\label{eq:ycommutators4.4}
[yxx(yyxx)^ny,y]=(q^2-q^{-2})\left(y(xxyy)^{n+1}-(yyxx)^{n+1}y\right). 
\end{equation}
\end{proposition}
\begin{proof}
Similar to the proof of Proposition \ref{prop:xcommutators1}. 
\end{proof}

\section{Doubly alternating and alternating words}
In this section, we express each doubly alternating word as a polynomial in the alternating words. This polynomial is with respect to the $q$-shuffle product. We start by recalling some notions. 

\begin{definition}\rm\label{def:derivatives}
(See \cite[Section 4]{PT}.) For a nontrivial word $w=a_1 \cdots a_n$, define 
\begin{equation*}
x^{-1}w=\begin{cases}a_2 \cdots a_n, & \text{ if }a_1=x;\\0, & \text{ if }a_1 \neq x.\end{cases}
\end{equation*}

\noindent By convention, $x^{-1}\m1=0$. 

\medskip
\noindent For $v \in \mV$, we define $x^{-1}v$ linearly. 

\medskip
\noindent For $v \in \mV$, we define $y^{-1}v$, $vx^{-1}$, $vy^{-1}$ in a similar way. 
\end{definition}

\noindent The above operators are essentially those appearing in \cite[(4.1)]{PT}. These operators were closely examined in \cite{PT}. 

\begin{lemma}\rm\label{lem:eqcond}
Let $v_1,v_2 \in \mV$. The following are equivalent. 
\begin{enumerate}
\item $v_1=v_2$; 
\item $v_1x^{-1}=v_2x^{-1}$ and $v_1y^{-1}=v_2y^{-1}$; 
\item $x^{-1}v_1=x^{-1}v_2$ and $y^{-1}v_1=y^{-1}v_2$. 
\end{enumerate}
\end{lemma}
\begin{proof}
Follows from Definition \ref{def:derivatives}. 
\end{proof}

\noindent Next we express each doubly alternating word as a polynomial in the alternating words. For notational convenience, we define $G_0=\m1$ and $\tG_0=\m1$. 

\begin{proposition}\rm\label{prop:conv1}
For $n \in \mN$, 

\begin{equation}\label{eq:conv1.1}
\sum_{k=0}^{2n}(-1)^k\tG_k \star \tG_{2n-k}=(-1)^n[2]_q^{2n}(xxyy)^n, 
\end{equation}
\begin{equation}\label{eq:conv1.2}
\sum_{k=0}^{2n}(-1)^kG_k \star G_{2n-k}=(-1)^n[2]_q^{2n}(yyxx)^n, 
\end{equation}
\begin{equation}\label{eq:conv1.3}
\sum_{k=0}^{2n}(-1)^kW_{-k} \star W_{k-2n}=(-1)^nq[2]_q^{2n+1}(xxyy)^nxx, 
\end{equation}
\begin{equation}\label{eq:conv1.4}
\sum_{k=0}^{2n}(-1)^kW_{k+1} \star W_{2n+1-k}=(-1)^nq[2]_q^{2n+1}(yyxx)^nyy. 
\end{equation}
\end{proposition}
\begin{proof}
We first show \eqref{eq:conv1.1} and \eqref{eq:conv1.3} by induction on $n$. 

\medskip
\noindent Clearly \eqref{eq:conv1.1} and \eqref{eq:conv1.3} hold for $n=0$. We now show the inductive steps. For $n \in \mN$ consider the following assertions: 
\begin{enumerate}
\item if \eqref{eq:conv1.3} holds for $n$, then \eqref{eq:conv1.1} holds for $n+1$; 
\item if \eqref{eq:conv1.1} holds for $n+1$, then \eqref{eq:conv1.3} holds for $n+1$. 
\end{enumerate}

\noindent We first show (i). Assume that \eqref{eq:conv1.3} holds for $n$. To show that \eqref{eq:conv1.1} holds for $n+1$, by Lemma \ref{lem:eqcond} it suffices to show that 
\begin{equation}\label{eq:conv1i1}
\sum_{k=0}^{2n+1}(-1)^k\tG_k \star W_{k-2n-1}+\sum_{k=1}^{2n+2}(-1)^kW_{1-k} \star \tG_{2n+2-k}=(-1)^{n+1}[2]_q^{2n+2}(xxyy)^nxxy. 
\end{equation}

\noindent To show \eqref{eq:conv1i1}, by Lemma \ref{lem:eqcond} it suffices to show that 
\begin{equation}\label{eq:conv1i2}
\sum_{k=0}^{2n+1}(-1)^k\tG_k \star \tG_{2n+1-k}+\sum_{k=1}^{2n+2}(-1)^k\tG_{k-1} \star \tG_{2n+2-k}=0
\end{equation}
and 
\begin{equation}\label{eq:conv1i3}
\sum_{k=1}^{2n+1}(-1)^k\left(q^{-2}W_{1-k} \star W_{k-2n-1}+W_{1-k} \star W_{k-2n-1}\right)=(-1)^{n+1}[2]_q^{2n+2}(xxyy)^nxx. 
\end{equation}
\noindent Note that \eqref{eq:conv1i2} can be checked by routine computation and \eqref{eq:conv1i3} follows from the inductive hypothesis. Consequently \eqref{eq:conv1i1} holds and \eqref{eq:conv1.1} holds for $n+1$. We have proved (i). 

\medskip
\noindent We now show (ii). Assume that \eqref{eq:conv1.1} holds for $n+1$. To show that \eqref{eq:conv1.3} holds for $n+1$, by Lemma \ref{lem:eqcond} it suffices to show that 
\begin{equation}\label{eq:conv1ii1}
\sum_{k=0}^{2n+2}(-1)^k\left(W_{-k} \star \tG_{2n+2-k}+q^2\tG_k \star W_{k-2n-2}\right)=(-1)^{n+1}q[2]_q^{2n+3}(xxyy)^{n+1}x. 
\end{equation}

\noindent To show \eqref{eq:conv1ii1}, by Lemma \ref{lem:eqcond} it suffices to show that 
\begin{equation}\label{eq:conv1ii2}
\sum_{k=0}^{2n+2}(-1)^k\left(\tG_k \star \tG_{2n+2-k}+q^2\tG_k \star \tG_{2n+2-k}\right)=(-1)^{n+1}q[2]_q^{2n+3}(xxyy)^{n+1}
\end{equation}
and 
\begin{equation}\label{eq:conv1ii3}
\sum_{k=0}^{2n+1}(-1)^kW_{-k} \star W_{k-2n-1}+\sum_{k=1}^{2n}(-1)^kW_{1-k} \star W_{k-2n-2}=0. 
\end{equation}

\noindent Note that \eqref{eq:conv1ii2} follows from the inductive hypothesis and \eqref{eq:conv1ii3} can be checked by routine computation. Consequently \eqref{eq:conv1ii1} holds and \eqref{eq:conv1.3} holds for $n+1$. We have proved (ii). 

\medskip
\noindent By the above discussions, we have proved \eqref{eq:conv1.1} and \eqref{eq:conv1.3}. Similarly, we can prove \eqref{eq:conv1.2} and \eqref{eq:conv1.4} using Lemma \ref{lem:eqcond} and induction on $n$. 
\end{proof}

\begin{proposition}\rm\label{prop:conv1odd}
For $n \in \mN$, 

\begin{equation}\label{eq:conv1.1odd}
\sum_{k=0}^{2n+1}(-1)^k\tG_k \star \tG_{2n+1-k}=0, 
\end{equation}
\begin{equation}\label{eq:conv1.2odd}
\sum_{k=0}^{2n+1}(-1)^kG_k \star G_{2n+1-k}=0, 
\end{equation}
\begin{equation}\label{eq:conv1.3odd}
\sum_{k=0}^{2n+1}(-1)^kW_{-k} \star W_{k-2n-1}=0, 
\end{equation}
\begin{equation}\label{eq:conv1.4odd}
\sum_{k=0}^{2n+1}(-1)^kW_{k+1} \star W_{2n+2-k}=0. 
\end{equation}
\end{proposition}
\begin{proof}
We first show \eqref{eq:conv1.1odd}. 

\medskip
\noindent By \cite[Proposition 5.10]{ter_alternating}, the words $\{\tG_k\}_{k \in \mN}$ mutually commute. Therefore, 
\begin{align*}
&\sum_{k=0}^{2n+1}(-1)^k\tG_k \star \tG_{2n+1-k}\\
&=\sum_{k=0}^{n}(-1)^k\tG_k \star \tG_{2n+1-k}+\sum_{k=n+1}^{2n+1}(-1)^k\tG_k \star \tG_{2n+1-k}\\
&=\sum_{k=0}^{n}(-1)^k\tG_k \star \tG_{2n+1-k}+\sum_{k=0}^{n}(-1)^{2n+1-k}\tG_{2n+1-k} \star \tG_k\\
&=\sum_{k=0}^{n}(-1)^k\tG_k \star \tG_{2n+1-k}-\sum_{k=0}^{n}(-1)^k\tG_k \star \tG_{2n+1-k}\\
&=0. 
\end{align*}

\noindent The remaining relations can be proved in a similar way. 
\end{proof}

\begin{proposition}\rm\label{prop:W-*tG}
For $n \in \mN$, 

\begin{equation}\label{eq:conv3.3}
\begin{split}
\sum_{k=0}^{2n}(-1)^kW_{-k} \star \tG_{2n-k}&=(-1)^n[2]_q^{2n}(xxyy)^nx\\
&=\sum_{k=0}^{2n}(-1)^k\tG_{2n-k} \star W_{-k}, 
\end{split}
\end{equation}
\begin{equation}\label{eq:conv3.1}
\begin{split}
q^{-1}\sum_{k=0}^{2n+1}(-1)^kW_{-k} \star \tG_{2n+1-k}&=(-1)^n[2]_q^{2n+1}(xxyy)^nxxy\\
&=q\sum_{k=0}^{2n+1}(-1)^k\tG_{2n+1-k} \star W_{-k}. 
\end{split}
\end{equation}
\end{proposition}
\begin{proof}
Note that 
\begin{align*}
\left(\sum_{k=0}^{2n}(-1)^kW_{-k} \star \tG_{2n-k}\right)x^{-1}=\sum_{k=0}^{2n}(-1)^k\tG_k \star \tG_{2n-k}=(-1)^n[2]_q^{2n}(xxyy)^n
\end{align*}
by \eqref{eq:conv1.1}, and 
\begin{align*}
\left(\sum_{k=0}^{2n}(-1)^kW_{-k} \star \tG_{2n-k}\right)y^{-1}=\sum_{k=0}^{2n-1}(-1)^kW_{-k} \star W_{k-2n+1}=0
\end{align*}
by \eqref{eq:conv1.3odd}. 

\medskip
\noindent By Lemma \ref{lem:eqcond}, we have proved the first equality in \eqref{eq:conv3.3}. The remaining equalities can be proved in a similar way. 
\end{proof}

\begin{proposition}\rm\label{prop:W-*G}
For $n \in \mN$, 

\begin{equation}\label{eq:conv2.3}
\begin{split}
\sum_{k=0}^{2n}(-1)^kW_{-k} \star G_{2n-k}&=(-1)^n[2]_q^{2n}x(yyxx)^n\\
&=\sum_{k=0}^{2n}(-1)^kG_{2n-k} \star W_{-k}, 
\end{split}
\end{equation}
\begin{equation}\label{eq:conv2.2}
\begin{split}
q\sum_{k=0}^{2n+1}(-1)^kW_{-k} \star G_{2n+1-k}&=(-1)^n[2]_q^{2n+1}yxx(yyxx)^n\\
&=q^{-1}\sum_{k=0}^{2n+1}(-1)^kG_{2n+1-k} \star W_{-k}. 
\end{split}
\end{equation}
\end{proposition}
\begin{proof}
Similar to the proof of Proposition \ref{prop:W-*tG}. 
\end{proof}

\begin{proposition}\rm\label{prop:W+*tG}
For $n \in \mN$, 

\begin{equation}\label{eq:conv2.4}
\begin{split}
\sum_{k=0}^{2n}(-1)^kW_{k+1} \star \tG_{2n-k}&=(-1)^n[2]_q^{2n}y(xxyy)^n\\
&=\sum_{k=0}^{2n}(-1)^k\tG_{2n-k} \star W_{k+1}, 
\end{split}
\end{equation}
\begin{equation}\label{eq:conv2.1}
\begin{split}
q\sum_{k=0}^{2n+1}(-1)^kW_{k+1} \star \tG_{2n+1-k}&=(-1)^n[2]_q^{2n+1}xyy(xxyy)^n\\
&=q^{-1}\sum_{k=0}^{2n+1}(-1)^k\tG_{2n+1-k} \star W_{k+1}. 
\end{split}
\end{equation}
\end{proposition}
\begin{proof}
Similar to the proof of Proposition \ref{prop:W-*tG}. 
\end{proof}

\begin{proposition}\rm\label{prop:W+*G}
For $n \in \mN$, 

\begin{equation}\label{eq:conv3.4}
\begin{split}
\sum_{k=0}^{2n}(-1)^kW_{k+1} \star G_{2n-k}&=(-1)^n[2]_q^{2n}(yyxx)^ny\\
&=\sum_{k=0}^{2n}(-1)^kG_{2n-k} \star W_{k+1}, 
\end{split}
\end{equation}
\begin{equation}\label{eq:conv3.2}
\begin{split}
q^{-1}\sum_{k=0}^{2n+1}(-1)^kW_{k+1} \star G_{2n+1-k}&=(-1)^n[2]_q^{2n+1}(yyxx)^nyyx\\
&=q\sum_{k=0}^{2n+1}(-1)^kG_{2n+1-k} \star W_{k+1}. 
\end{split}
\end{equation}
\end{proposition}
\begin{proof}
Similar to the proof of Proposition \ref{prop:W-*tG}. 
\end{proof}

\begin{proposition}\rm\label{prop:G*tG}
For $n \in \mN$, 

\begin{equation}\label{eq:conv4.7}
\sum_{k=0}^{2n+2}(-1)^kG_k \star \tG_{2n+2-k}=(-1)^{n+1}[2]_q^{2n+1}\big(q^{-1}xyy(xxyy)^nx+qy(xxyy)^nxxy\big), 
\end{equation}
\begin{equation}\label{eq:conv4.8}
\sum_{k=0}^{2n+2}(-1)^k\tG_{2n+2-k} \star G_k=(-1)^{n+1}[2]_q^{2n+1}\big(qxyy(xxyy)^nx+q^{-1}y(xxyy)^nxxy\big), 
\end{equation}
\begin{equation}\label{eq:conv4.1}
\sum_{k=0}^{2n+1}(-1)^kG_k \star \tG_{2n+1-k}=(-1)^n[2]_q^{2n}\big(x(yyxx)^ny-y(xxyy)^nx\big), 
\end{equation}
\begin{equation}\label{eq:conv4.2}
\sum_{k=0}^{2n+1}(-1)^k\tG_{2n+1-k} \star G_k=(-1)^n[2]_q^{2n}\big(x(yyxx)^ny-y(xxyy)^nx\big). 
\end{equation}
\end{proposition}
\begin{proof}
Similar to the proof of Proposition \ref{prop:W-*tG}. 
\end{proof}

\begin{proposition}\rm\label{prop:W+*W-}
For $n \in \mN$, 

\begin{equation}\label{eq:conv4.3}
\sum_{k=0}^{2n}(-1)^kW_{k+1} \star W_{k-2n}=(-1)^n[2]_q^{2n}\big(q^{-2}x(yyxx)^ny+y(xxyy)^nx\big), 
\end{equation}
\begin{equation}\label{eq:conv4.4}
\sum_{k=0}^{2n}(-1)^kW_{k-2n} \star W_{k+1}=(-1)^n[2]_q^{2n}\big(x(yyxx)^ny+q^{-2}y(xxyy)^nx\big), 
\end{equation}
\begin{equation}\label{eq:conv4.5}
\sum_{k=0}^{2n+1}(-1)^kW_{k+1} \star W_{k-2n-1}=(-1)^nq^{-1}[2]_q^{2n+1}\big(xyy(xxyy)^nx-yxx(yyxx)^ny\big), 
\end{equation}
\begin{equation}\label{eq:conv4.6}
\sum_{k=0}^{2n+1}(-1)^kW_{k-2n-1} \star W_{k+1}=(-1)^nq^{-1}[2]_q^{2n+1}\big(xyy(xxyy)^nx-yxx(yyxx)^ny\big). 
\end{equation}
\end{proposition}
\begin{proof}
Similar to the proof of Proposition \ref{prop:W-*tG}. 
\end{proof}

\noindent Note that \eqref{eq:conv1.1}--\eqref{eq:conv1.4} and \eqref{eq:conv3.3}--\eqref{eq:conv3.2} express the doubly alternating words listed in \eqref{eq:dalternating1}--\eqref{eq:dalternating3} as polynomials in the alternating words. Using \eqref{eq:conv4.7}--\eqref{eq:conv4.6}, next we express the doubly alternating words listed in \eqref{eq:dalternating4} as polynomials in the alternating words. 

\begin{corollary}\rm\label{cor:solve4.1&4.2}
Let $n \in \mN$. We have 

\begin{equation}\label{eq:solve4.1}
(-1)^n[2]_q^{2n+1}x(yyxx)^ny=(1-q^{-2})^{-1}\sum_{k=0}^{2n}(-1)^k[W_{k-2n},W_{k+1}]_q, 
\end{equation}
\begin{equation}\label{eq:solve4.2}
(-1)^n[2]_q^{2n+1}y(xxyy)^nx=(1-q^{-2})^{-1}\sum_{k=0}^{2n}(-1)^k[W_{k+1},W_{k-2n}]_q, 
\end{equation}
\begin{equation}\label{eq:solve4.1&4.2}
\begin{split}
\sum_{k=0}^{2n+1}(-1)^kG_k \star \tG_{2n+1-k}&=(1-q^{-2})^{-1}\sum_{k=0}^{2n}(-1)^k[W_{k-2n},W_{k+1}]\\
&=\sum_{k=0}^{2n+1}(-1)^k\tG_{2n+1-k} \star G_k. 
\end{split}
\end{equation}
\end{corollary}
\begin{proof}
Follows from \eqref{eq:conv4.1}--\eqref{eq:conv4.4}. 
\end{proof}

\begin{corollary}\rm\label{cor:solve4.3&4.4}
Let $n \in \mN$. We have 

\begin{equation}\label{eq:solve4.3}
(-1)^{n+1}[2]_q^{2n+2}xyy(xxyy)^nx=(q-q^{-1})^{-1}\sum_{k=0}^{2n+2}(-1)^k[\tG_{2n+2-k},G_k]_q, 
\end{equation}
\begin{equation}\label{eq:solve4.4}
(-1)^{n+1}[2]_q^{2n+2}yxx(yyxx)^ny=(q-q^{-1})^{-1}\sum_{k=0}^{2n+2}(-1)^k[G_k,\tG_{2n+2-k}]_q. 
\end{equation}
\begin{equation}\label{eq:solve4.3&4.4}
\begin{split}
\sum_{k=0}^{2n+1}(-1)^kW_{k+1} \star W_{k-2n-1}&=(q^2-1)^{-1}\sum_{k=0}^{2n+2}(-1)^k[G_k,\tG_{2n+2-k}]\\
&=\sum_{k=0}^{2n+1}(-1)^kW_{k-2n-1} \star W_{k+1}. 
\end{split}
\end{equation}
\end{corollary}
\begin{proof}
Follows from \eqref{eq:conv4.7}, \eqref{eq:conv4.8}, \eqref{eq:conv4.5}, \eqref{eq:conv4.6}. 
\end{proof}

\section{Generating functions}
In this section, we express the results from Section 5 in terms of generating functions. For the rest of this paper, let $t$ denote an indeterminate. 

\begin{proposition}\rm\label{prop:genfun1}
We have 

\begin{equation}\label{eq:genfun1.1}
\tG(-t) \star \tG(t)=\sum_{n \in \mN}(-1)^n[2]_q^{2n}(xxyy)^nt^{2n}, 
\end{equation}
\begin{equation}\label{eq:genfun1.2}
G(-t) \star G(t)=\sum_{n \in \mN}(-1)^n[2]_q^{2n}(yyxx)^nt^{2n}, 
\end{equation}
\begin{equation}\label{eq:genfun1.3}
W^-(-t) \star W^-(t)=\sum_{n \in \mN}(-1)^nq[2]_q^{2n+1}(xxyy)^nxxt^{2n}, 
\end{equation}
\begin{equation}\label{eq:genfun1.4}
W^+(-t) \star W^+(t)=\sum_{n \in \mN}(-1)^nq[2]_q^{2n+1}(yyxx)^nyyt^{2n}. 
\end{equation}
\end{proposition}
\begin{proof}
This is Propositions \ref{prop:conv1} and \ref{prop:conv1odd} in terms of generating functions. 
\end{proof}

\begin{proposition}\rm\label{prop:W-*tGgenfun}
We have 

\begin{equation}\label{eq:genfun3.1}
\begin{split}
W^-(-t) \star \tG(t)&=\sum_{n \in \mN}(-1)^n[2]_q^{2n}(xxyy)^nxt^{2n}\\
&+q\sum_{n \in \mN}(-1)^n[2]_q^{2n+1}(xxyy)^nxxyt^{2n+1}, 
\end{split}
\end{equation}
\begin{equation}\label{eq:genfun3.2}
\begin{split}
\tG(t) \star W^-(-t)&=\sum_{n \in \mN}(-1)^n[2]_q^{2n}(xxyy)^nxt^{2n}\\
&+q^{-1}\sum_{n \in \mN}(-1)^n[2]_q^{2n+1}(xxyy)^nxxyt^{2n+1}. 
\end{split}
\end{equation}
\end{proposition}
\begin{proof}
This is Propositions \ref{prop:W-*tG} in terms of generating functions. 
\end{proof}

\begin{proposition}\rm\label{prop:W-*Ggenfun}
We have 

\begin{equation}\label{eq:genfun2.3}
\begin{split}
W^-(-t) \star G(t)&=\sum_{n \in \mN}(-1)^n[2]_q^{2n}x(yyxx)^nt^{2n}\\
&+q^{-1}\sum_{n \in \mN}(-1)^n[2]_q^{2n+1}yxx(yyxx)^nt^{2n+1}, 
\end{split}
\end{equation}
\begin{equation}\label{eq:genfun2.4}
\begin{split}
G(t) \star W^-(-t)&=\sum_{n \in \mN}(-1)^n[2]_q^{2n}x(yyxx)^nt^{2n}\\
&+q\sum_{n \in \mN}(-1)^n[2]_q^{2n+1}yxx(yyxx)^nt^{2n+1}. 
\end{split}
\end{equation}
\end{proposition}
\begin{proof}
This is Propositions \ref{prop:W-*G} in terms of generating functions. 
\end{proof}

\begin{proposition}\rm\label{prop:W+*tGgenfun}
We have 

\begin{equation}\label{eq:genfun2.1}
\begin{split}
W^+(-t) \star \tG(t)&=\sum_{n \in \mN}(-1)^n[2]_q^{2n}y(xxyy)^nt^{2n}\\
&+q^{-1}\sum_{n \in \mN}(-1)^n[2]_q^{2n+1}xyy(xxyy)^nt^{2n+1}, 
\end{split}
\end{equation}
\begin{equation}\label{eq:genfun2.2}
\begin{split}
\tG(t) \star W^+(-t)&=\sum_{n \in \mN}(-1)^n[2]_q^{2n}y(xxyy)^nt^{2n}\\
&+q\sum_{n \in \mN}(-1)^n[2]_q^{2n+1}xyy(xxyy)^nt^{2n+1}. 
\end{split}
\end{equation}
\end{proposition}
\begin{proof}
This is Propositions \ref{prop:W+*tG} in terms of generating functions. 
\end{proof}

\begin{proposition}\rm\label{prop:W+*Ggenfun}
We have 

\begin{equation}\label{eq:genfun3.3}
\begin{split}
W^+(-t) \star G(t)&=\sum_{n \in \mN}(-1)^n[2]_q^{2n}(yyxx)^nyt^{2n}\\
&+q\sum_{n \in \mN}(-1)^n[2]_q^{2n+1}(yyxx)^nyyxt^{2n+1}, 
\end{split}
\end{equation}
\begin{equation}\label{eq:genfun3.4}
\begin{split}
G(t) \star W^+(-t)&=\sum_{n \in \mN}(-1)^n[2]_q^{2n}(yyxx)^nyt^{2n}\\
&+q^{-1}\sum_{n \in \mN}(-1)^n[2]_q^{2n+1}(yyxx)yyx^nt^{2n+1}. 
\end{split}
\end{equation}
\end{proposition}
\begin{proof}
This is Propositions \ref{prop:W+*G} in terms of generating functions. 
\end{proof}

\begin{proposition}\rm\label{prop:G*tGgenfun}
We have 

\begin{equation}\label{eq:genfun4.1}
\begin{split}
G(-t) \star \tG(t)&=\sum_{n \in \mN}(-1)^n[2]_q^{2n-1}\big(q^{-1}xyy(xxyy)^{n-1}x+qy(xxyy)^{n-1}xxy\big)t^{2n}\\
&+\sum_{n \in \mN}(-1)^n[2]_q^{2n}\big(x(yyxx)^ny-y(xxyy)^nx\big)t^{2n+1}, 
\end{split}
\end{equation}
\begin{equation}\label{eq:genfun4.2}
\begin{split}
\tG(t) \star G(-t)&=\sum_{n \in \mN}(-1)^n[2]_q^{2n-1}\big(qxyy(xxyy)^{n-1}x+q^{-1}y(xxyy)^{n-1}xxy\big)t^{2n}\\
&+\sum_{n \in \mN}(-1)^n[2]_q^{2n}\big(x(yyxx)^ny-y(xxyy)^nx\big)t^{2n+1}. 
\end{split}
\end{equation}
\end{proposition}
\begin{proof}
This is Proposition \ref{prop:G*tG} in terms of generating functions. 
\end{proof}

\begin{proposition}\rm\label{prop:W+*W-genfun}
We have 

\begin{equation}\label{eq:genfun4.3}
\begin{split}
W^+(-t) \star W^-(t)&=\sum_{n \in \mN}(-1)^n[2]_q^{2n}\big(q^{-2}x(yyxx)^ny+y(xxyy)^nx\big)t^{2n}\\
&+q^{-1}\sum_{n \in \mN}(-1)^n[2]_q^{2n+1}\big(xyy(xxyy)^nx-yxx(yyxx)^ny\big)t^{2n+1}, 
\end{split}
\end{equation}
\begin{equation}\label{eq:genfun4.4}
\begin{split}
W^-(t) \star W^+(-t)&=\sum_{n \in \mN}(-1)^n[2]_q^{2n}\big(x(yyxx)^ny+q^{-2}y(xxyy)^nx\big)t^{2n}\\
&+q^{-1}\sum_{n \in \mN}(-1)^n[2]_q^{2n+1}\big(xyy(xxyy)^nx-yxx(yyxx)^ny\big)t^{2n+1}. 
\end{split}
\end{equation}
\end{proposition}
\begin{proof}
This is Proposition \ref{prop:W+*W-} in terms of generating functions. 
\end{proof}

\section{Appendix A}
In this appendix, we list some known commutator relations on the alternating words. These commutator relations were obtained by Terwilliger in \cite[Propositions 5.10 and 5.11]{ter_alternating}. We remark that these commutator relations are with respect to the $q$-shuffle product. 

\medskip
\noindent Let $i,j \in \mN$. The following relations hold in $U$. 
\begin{align}
&[W_{-i},W_{-j}]=0, \hspace{2.4em} [W_{i+1},W_{j+1}]=0, \\
&[G_{i+1},G_{j+1}]=0, \hspace{2em} [\tG_{i+1},\tG_{j+1}]=0, \\
&[W_{-i},W_{j+1}]=[W_{-j},W_{i+1}], \\
&[W_{-i},G_{j+1}]=[W_{-j},G_{i+1}], \\
&[W_{-i},\tG_{j+1}]=[W_{-j},\tG_{i+1}], \\
&[W_{i+1},G_{j+1}]=[W_{j+1},G_{i+1}], \\
&[W_{i+1},\tG_{j+1}]=[W_{j+1},\tG_{i+1}], \\
&[G_{i+1},\tG_{j+1}]=[G_{j+1},\tG_{i+1}], \\
&[W_{-i},G_j]_q=[W_{-j},G_i]_q, \hspace{2em} [G_i,W_{j+1}]_q=[G_j,W_{i+1}]_q, \\
&[\tG_i,W_{-j}]_q=[\tG_j,W_{-i}]_q, \hspace{2em} [W_{i+1},\tG_j]_q=[W_{j+1},\tG_i]_q, \\
&[G_i,\tG_{j+1}]-[G_j,\tG_{i+1}]=q[W_{-j},W_{i+1}]_q-q[W_{-i},W_{j+1}]_q, \\
&[\tG_i,G_{j+1}]-[\tG_j,G_{i+1}]=q[W_{j+1},W_{-i}]_q-q[W_{i+1},W_{-j}]_q, \\
&[G_{i+1},\tG_{j+1}]_q-[G_{j+1},\tG_{i+1}]_q=q[W_{-j},W_{i+2}]-q[W_{-i},W_{j+2}], \\
&[\tG_{i+1},G_{j+1}]_q-[\tG_{j+1},G_{i+1}]_q=q[W_{j+1},W_{-i-1}]-q[W_{i+1},W_{-j-1}]. 
\end{align}

\section{Appendix B}
Recall that in Section 4 we give a commutator relation for each choice of one letter and one doubly alternating word. It is then natural to ask what are the commutator relations involving two doubly alternating words. Considering that there are $16$ families of doubly alternating words, the number of such commutator relations would be large. Furthermore, most of these commutator relations turn out to be complicated. In this appendix, we present some simple commutator relations. The readers may use the operators $x^{-1}$, $y^{-1}$ from Definition \ref{def:derivatives} to obtain the more complicated commutator relations. 

\begin{proposition}\rm\label{prop:commutators1}
Let $i,j \in \mN$. The following relations hold in $U$. 

\begin{equation}\label{eq:commutators1.1}
[(xxyy)^i,(xxyy)^j]=0, 
\end{equation}
\begin{equation}\label{eq:commutators1.2}
[(yyxx)^i,(yyxx)^j]=0, 
\end{equation}
\begin{equation}\label{eq:commutators1.3}
[(xxyy)^ixx,(xxyy)^jxx]=0, 
\end{equation}
\begin{equation}\label{eq:commutators1.4}
[(yyxx)^iyy,(yyxx)^jyy]=0. 
\end{equation}
\end{proposition}
\begin{proof}
By \eqref{eq:conv1.1}, the elements $(xxyy)^i,(xxyy)^j$ are polynomials in the elements $\{\tG_{n+1}\}_{n \in \mN}$. By \cite[Proposition 5.10]{ter_alternating}, the elements $\{\tG_{n+1}\}_{n \in \mN}$ mutually commutes. We obtain \eqref{eq:commutators1.1}. 

\medskip
\noindent The remaining relations can be proved in a similar way. 
\end{proof}

\begin{proposition}\rm\label{prop:commutators2}
Let $i,j \in \mN$. The following relations hold in $U$. 

\begin{equation}\label{eq:commutators2.1}
[(xxyy)^ixxy,(xxyy)^{j+1}]=[(xxyy)^jxxy,(xxyy)^{i+1}], 
\end{equation}
\begin{equation}\label{eq:commutators2.2}
[xyy(xxyy)^i,(xxyy)^{j+1}]=[xyy(xxyy)^j,(xxyy)^{i+1}], 
\end{equation}
\begin{equation}\label{eq:commutators2.3}
[(yyxx)^iyyx,(yyxx)^{j+1}]=[(yyxx)^jyyx,(yyxx)^{i+1}], 
\end{equation}
\begin{equation}\label{eq:commutators2.4}
[yxx(yyxx)^i,(yyxx)^{j+1}]=[yxx(yyxx)^j,(yyxx)^{i+1}]. 
\end{equation}
\end{proposition}
\begin{proof}
To obtain \eqref{eq:commutators2.1}, on both sides of \eqref{eq:commutators1.1} apply $y^{-1}$ to the right. 

\medskip
\noindent The remaining relations can be proved in a similar way. 
\end{proof}

\begin{proposition}\rm\label{prop:commutators3}
Let $i,j \in \mN$. The following relations hold in $U$. 

\begin{equation}\label{eq:commutators3.1}
[(xxyy)^ix,(xxyy)^{j+1}]=[(xxyy)^jx,(xxyy)^{i+1}], 
\end{equation}
\begin{equation}\label{eq:commutators3.2}
[y(xxyy)^i,(xxyy)^{j+1}]=[y(xxyy)^j,(xxyy)^{i+1}], 
\end{equation}
\begin{equation}\label{eq:commutators3.3}
[(yyxx)^iy,(yyxx)^{j+1}]=[(yyxx)^jy,(yyxx)^{i+1}], 
\end{equation}
\begin{equation}\label{eq:commutators3.4}
[x(yyxx)^i,(yyxx)^{j+1}]=[x(yyxx)^j,(yyxx)^{i+1}]. 
\end{equation}
\end{proposition}
\begin{proof}
To obtain \eqref{eq:commutators3.1}, on both sides of \eqref{eq:commutators2.1} apply $y^{-1}x^{-1}$ to the right. 

\medskip
\noindent The remaining relations can be proved in a similar way. 
\end{proof}

\begin{proposition}\rm\label{prop:commutators4}
Let $i,j \in \mN$. The following relations hold in $U$. 

\begin{equation}\label{eq:commutators4.1}
[(xxyy)^ixx,(xxyy)^jxxy]=[(xxyy)^jxx,(xxyy)^ixxy], 
\end{equation}
\begin{equation}\label{eq:commutators4.2}
[(yyxx)^iyy,xyy(xxyy)^j]=[(yyxx)^jyy,xyy(xxyy)^i], 
\end{equation}
\begin{equation}\label{eq:commutators4.3}
[(yyxx)^iyy,(yyxx)^jyyx]=[(yyxx)^jyy,(yyxx)^iyyx], 
\end{equation}
\begin{equation}\label{eq:commutators4.4}
[(xxyy)^ixx,yxx(yyxx)^j]=[(xxyy)^jxx,yxx(yyxx)^i]. 
\end{equation}
\end{proposition}
\begin{proof}
To obtain \eqref{eq:commutators4.1}, on both sides of \eqref{eq:commutators2.1} apply $y^{-1}y^{-1}$ to the right. 

\medskip
\noindent The remaining relations can be proved in a similar way. 
\end{proof}

\begin{proposition}\rm\label{prop:commutators5}
Let $i,j \in \mN$. The following relations hold in $U$. 

\begin{equation}\label{eq:commutators5.1}
[(xxyy)^ix,(xxyy)^jxxy]_q=[(xxyy)^jx,(xxyy)^ixxy]_q, 
\end{equation}
\begin{equation}\label{eq:commutators5.2}
[xyy(xxyy)^i,y(xxyy)^j]_q=[xyy(xxyy)^j,y(xxyy)^i]_q, 
\end{equation}
\begin{equation}\label{eq:commutators5.3}
[(yyxx)^iy,(yyxx)^jyyx]_q=[(yyxx)^jy,(yyxx)^iyyx]_q, 
\end{equation}
\begin{equation}\label{eq:commutators5.4}
[yxx(yyxx)^i,x(yyxx)^j]_q=[yxx(yyxx)^j,x(yyxx)^i]_q. 
\end{equation}
\end{proposition}
\begin{proof}
To obtain \eqref{eq:commutators5.1}, on both sides of \eqref{eq:commutators4.1} apply $x^{-1}$ to the right. 

\medskip
\noindent The remaining relations can be proved in a similar way. 
\end{proof}

\begin{proposition}\rm\label{prop:commutators6}
Let $i,j \in \mN$. The following relations hold in $U$. 

\begin{equation}\label{eq:commutators6.1}
[(xxyy)^i,(xxyy)^jxxy]_{q^2}=[(xxyy)^j,(xxyy)^ixxy]_{q^2}, 
\end{equation}
\begin{equation}\label{eq:commutators6.2}
[xyy(xxyy)^i,(xxyy)^j]_{q^2}=[xyy(xxyy)^j,(xxyy)^i]_{q^2}, 
\end{equation}
\begin{equation}\label{eq:commutators6.3}
[(yyxx)^i,(yyxx)^jyyx]_{q^2}=[(yyxx)^j,(yyxx)^iyyx]_{q^2}, 
\end{equation}
\begin{equation}\label{eq:commutators6.4}
[yxx(yyxx)^i,(yyxx)^j]_{q^2}=[yxx(yyxx)^j,(yyxx)^i]_{q^2}. 
\end{equation}
\end{proposition}
\begin{proof}
To obtain \eqref{eq:commutators6.1}, on both sides of \eqref{eq:commutators5.1} apply $x^{-1}$ to the right. 

\medskip
\noindent The remaining relations can be proved in a similar way. 
\end{proof}

\begin{proposition}\rm\label{prop:commutators7}
Let $i,j \in \mN$. The following relations hold in $U$. 

\begin{equation}\label{eq:commutators7.1}
[(xxyy)^ix,(xxyy)^jxx]_{q^2}=[(xxyy)^jx,(xxyy)^ixx]_{q^2}, 
\end{equation}
\begin{equation}\label{eq:commutators7.2}
[(xxyy)^ixx,x(yyxx)^j]_{q^2}=[(xxyy)^jxx,x(yyxx)^i]_{q^2}, 
\end{equation}
\begin{equation}\label{eq:commutators7.3}
[(yyxx)^iy,(yyxx)^jyy]_{q^2}=[(yyxx)^jy,(yyxx)^iyy]_{q^2}, 
\end{equation}
\begin{equation}\label{eq:commutators7.4}
[(yyxx)^iyy,y(xxyy)^j]_{q^2}=[(yyxx)^jyy,y(xxyy)^i]_{q^2}. 
\end{equation}
\end{proposition}
\begin{proof}
To obtain \eqref{eq:commutators7.1}, on both sides of \eqref{eq:commutators1.3} apply $x^{-1}$ to the right. 

\medskip
\noindent The remaining relations can be proved in a similar way. 
\end{proof}

\begin{proposition}\rm\label{prop:commutators8}
Let $i,j \in \mN$. The following relations hold in $U$. 

\begin{equation}\label{eq:commutators8.1}
[(xxyy)^ixxy,(xxyy)^{j+1}xx]_{q^2}=[(xxyy)^jxxy,(xxyy)^{i+1}xx]_{q^2}, 
\end{equation}
\begin{equation}\label{eq:commutators8.2}
[(xxyy)^{i+1}xx,yxx(yyxx)^j]_{q^2}=[(xxyy)^{j+1}xx,yxx(yyxx)^i]_{q^2}, 
\end{equation}
\begin{equation}\label{eq:commutators8.3}
[(yyxx)^iyyx,(yyxx)^{j+1}yy]_{q^2}=[(yyxx)^jyyx,(yyxx)^{i+1}yy]_{q^2}, 
\end{equation}
\begin{equation}\label{eq:commutators8.4}
[(yyxx)^{i+1}yy,xyy(xxyy)^j]_{q^2}=[(yyxx)^{j+1}yy,xyy(xxyy)^i]_{q^2}. 
\end{equation}
\end{proposition}
\begin{proof}
To obtain \eqref{eq:commutators8.1}, on both sides of \eqref{eq:commutators7.1} apply $x^{-1}y^{-1}$ to the right. 

\medskip
\noindent The remaining relations can be proved in a similar way. 
\end{proof}

\begin{proposition}\rm\label{prop:commutators9}
Let $i,j \in \mN$. The following relations hold in $U$. 

\begin{equation}\label{eq:commutators9.1}
[(xxyy)^i,(xxyy)^jx]_{q^2}=[(xxyy)^j,(xxyy)^ix]_{q^2}, 
\end{equation}
\begin{equation}\label{eq:commutators9.2}
[x(yyxx)^i,(yyxx)^j]_{q^2}=[x(yyxx)^j,(yyxx)^i]_{q^2}, 
\end{equation}
\begin{equation}\label{eq:commutators9.3}
[(yyxx)^i,(yyxx)^jy]_{q^2}=[(yyxx)^j,(yyxx)^iy]_{q^2}, 
\end{equation}
\begin{equation}\label{eq:commutators9.4}
[y(xxyy)^i,(xxyy)^j]_{q^2}=[y(xxyy)^j,(xxyy)^i]_{q^2}. 
\end{equation}
\end{proposition}
\begin{proof}
To obtain \eqref{eq:commutators9.1}, on both sides of \eqref{eq:commutators7.1} apply $x^{-1}x^{-1}$ to the right. 

\medskip
\noindent The remaining relations can be proved in a similar way. 
\end{proof}

\begin{proposition}\rm\label{prop:commutators10}
Let $i,j \in \mN$. The following relations hold in $U$. 

\begin{equation}\label{eq:commutators10.1}
[(xxyy)^ixxy,(xxyy)^{j+1}x]_q=[(xxyy)^jxxy,(xxyy)^{i+1}x]_q, 
\end{equation}
\begin{equation}\label{eq:commutators10.2}
[x(yyxx)^{i+1},yxx(yyxx)^j]_q=[x(yyxx)^{j+1},yxx(yyxx)^i]_q, 
\end{equation}
\begin{equation}\label{eq:commutators10.3}
[(yyxx)^iyyx,(yyxx)^{j+1}y]_q=[(yyxx)^jyyx,(yyxx)^{i+1}y]_q, 
\end{equation}
\begin{equation}\label{eq:commutators10.4}
[y(xxyy)^{i+1},xyy(xxyy)^j]_q=[y(xxyy)^{j+1},xyy(xxyy)^i]_q. 
\end{equation}
\end{proposition}
\begin{proof}
To obtain \eqref{eq:commutators10.1}, on both sides of \eqref{eq:commutators9.1} apply $y^{-1}$ to the right. 

\medskip
\noindent The remaining relations can be proved in a similar way. 
\end{proof}

\begin{proposition}\rm\label{prop:commutators11}
Let $i,j \in \mN$. The following relations hold in $U$. 

\begin{equation}\label{eq:commutators11.1}
[(xxyy)^ixx,(xxyy)^{j+1}x]=[(xxyy)^jxx,(xxyy)^{i+1}x], 
\end{equation}
\begin{equation}\label{eq:commutators11.2}
[x(yyxx)^{i+1},xx(yyxx)^j]=[x(yyxx)^{j+1},xx(yyxx)^i], 
\end{equation}
\begin{equation}\label{eq:commutators11.3}
[(yyxx)^iyy,(yyxx)^{j+1}y]=[(yyxx)^jyy,(yyxx)^{i+1}y], 
\end{equation}
\begin{equation}\label{eq:commutators11.4}
[y(xxyy)^{i+1},yy(xxyy)^j]=[y(xxyy)^{j+1},yy(xxyy)^i]. 
\end{equation}
\end{proposition}
\begin{proof}
To obtain \eqref{eq:commutators11.1}, on both sides of \eqref{eq:commutators10.1} apply $y^{-1}$ to the right. 

\medskip
\noindent The remaining relations can be proved in a similar way. 
\end{proof}

\section{Appendix C}
Recall the formulas \eqref{eq:conv3.3}--\eqref{eq:conv3.2}, \eqref{eq:solve4.1&4.2}, \eqref{eq:solve4.3&4.4}. These formulas give new relations in the alternating words. In this appendix, we give alternative proofs to them without using the doubly alternating words. 

\begin{proposition}\rm\label{prop:W-*tGapp}
For $n \in \mN$, 
\begin{equation}\label{eq:conv3.3app}
\sum_{k=0}^{2n}(-1)^kW_{-k} \star \tG_{2n-k}=\sum_{k=0}^{2n}(-1)^k\tG_{2n-k} \star W_{-k}, 
\end{equation}
\begin{equation}\label{eq:conv3.1app}
q^{-1}\sum_{k=0}^{2n+1}(-1)^kW_{-k} \star \tG_{2n+1-k}=q\sum_{k=0}^{2n+1}(-1)^k\tG_{2n+1-k} \star W_{-k}. 
\end{equation}
\end{proposition}
\begin{proof}
By \cite[(43)]{ter_alternating} we have 
\begin{equation*}
\sum_{k=0}^{n-1}(-1)^k\left([W_{-k},\tG_{2n-k}]+[\tG_{k+1},W_{k+1-2n}]\right)=0. 
\end{equation*}

\noindent Rearranging the terms yields \eqref{eq:conv3.3app}. 

\medskip
\noindent Similarly \eqref{eq:conv3.1app} follows from \cite[(49)]{ter_alternating}. 
\end{proof}

\begin{proposition}\rm\label{prop:W-*Gapp}
For $n \in \mN$, 
\begin{equation}\label{eq:conv2.3app}
\sum_{k=0}^{2n}(-1)^kW_{-k} \star G_{2n-k}=\sum_{k=0}^{2n}(-1)^kG_{2n-k} \star W_{-k}, 
\end{equation}
\begin{equation}\label{eq:conv2.2app}
q\sum_{k=0}^{2n+1}(-1)^kW_{-k} \star G_{2n+1-k}=q^{-1}\sum_{k=0}^{2n+1}(-1)^kG_{2n+1-k} \star W_{-k}. 
\end{equation}
\end{proposition}
\begin{proof}
Similar to the proof of Proposition \ref{prop:W-*tGapp}. 
\end{proof}

\begin{proposition}\rm\label{prop:W+*tGapp}
For $n \in \mN$, 
\begin{equation}\label{eq:conv2.4app}
\sum_{k=0}^{2n}(-1)^kW_{k+1} \star \tG_{2n-k}=\sum_{k=0}^{2n}(-1)^k\tG_{2n-k} \star W_{k+1}, 
\end{equation}
\begin{equation}\label{eq:conv2.1app}
q\sum_{k=0}^{2n+1}(-1)^kW_{k+1} \star \tG_{2n+1-k}=q^{-1}\sum_{k=0}^{2n+1}(-1)^k\tG_{2n+1-k} \star W_{k+1}. 
\end{equation}
\end{proposition}
\begin{proof}
Similar to the proof of Proposition \ref{prop:W-*tGapp}. 
\end{proof}

\begin{proposition}\rm\label{prop:W+*Gapp}
For $n \in \mN$, 
\begin{equation}\label{eq:conv3.4app}
\sum_{k=0}^{2n}(-1)^kW_{k+1} \star G_{2n-k}=\sum_{k=0}^{2n}(-1)^kG_{2n-k} \star W_{k+1}, 
\end{equation}
\begin{equation}\label{eq:conv3.2app}
q^{-1}\sum_{k=0}^{2n+1}(-1)^kW_{k+1} \star G_{2n+1-k}=q\sum_{k=0}^{2n+1}(-1)^kG_{2n+1-k} \star W_{k+1}. 
\end{equation}
\end{proposition}
\begin{proof}
Similar to the proof of Proposition \ref{prop:W-*tGapp}. 
\end{proof}

\begin{proposition}\rm\label{prop:G*tGapp}
For $n \in \mN$, 
\begin{equation}\label{eq:G*tGapp}
\begin{split}
\sum_{k=0}^{2n+1}(-1)^kG_k \star \tG_{2n+1-k}&=(1-q^{-2})^{-1}\sum_{k=0}^{2n}(-1)^k[W_{k-2n},W_{k+1}]\\
&=\sum_{k=0}^{2n+1}(-1)^k\tG_{2n+1-k} \star G_k. 
\end{split}
\end{equation}
\end{proposition}
\begin{proof}
By \cite[(47)]{ter_alternating} we have 
\begin{equation*}
\sum_{k=0}^n(-1)^k\left([\tG_k,G_{2n+1-k}]+[G_k,\tG_{2n+1-k}]\right)=0. 
\end{equation*}

\noindent Rearranging the terms yields 
\begin{equation}\label{eq:G*tGapp1}
\sum_{k=0}^{2n+1}(-1)^kG_k \star \tG_{2n+1-k}=\sum_{k=0}^{2n+1}(-1)^k\tG_{2n+1-k} \star G_k. 
\end{equation}

\noindent By \cite[(52)]{ter_alternating} we have 
\begin{align*}
&\sum_{k=0}^{n-1}(-1)^{k+1}\left([G_{k+1},\tG_{2n-k}]_q-[G_{2n-k},\tG_{k+1}]_q\right)\\
&\hspace{8em}=q\sum_{k=0}^{n-1}(-1)^{k+1}\Big([W_{k+1-2n},W_{k+2}]-[W_{-k},W_{2n+1-k}]\Big). 
\end{align*}

\noindent Rearrange the terms and apply \cite[(37)]{ter_alternating}. This yields 
\begin{equation}\label{eq:G*tGapp2}
\sum_{k=0}^{2n+1}(-1)^k[G_k,\tG_{2n+1-k}]_q=q\sum_{k=0}^{2n}(-1)^k[W_{k-2n},W_{k+1}]. 
\end{equation}

\noindent The result follows from \eqref{eq:G*tGapp1} and \eqref{eq:G*tGapp2}. 
\end{proof}

\begin{proposition}\rm\label{prop:W+*W-app}
For $n \in \mN$, 
\begin{equation}\label{eq:W+*W-app}
\begin{split}
\sum_{k=0}^{2n+1}(-1)^kW_{k+1} \star W_{k-2n-1}&=(q^2-1)^{-1}\sum_{k=0}^{2n+2}(-1)^k[G_k,\tG_{2n+2-k}]\\
&=\sum_{k=0}^{2n+1}(-1)^kW_{k-2n-1} \star W_{k+1}. 
\end{split}
\end{equation}
\end{proposition}
\begin{proof}
Similar to the proof of Proposition \ref{prop:G*tGapp}. 
\end{proof}

\section{Acknowledgement}
The majority of this paper was written while the author was a Math Ph.D. student at the University of Wisconsin-Madison. Most results in this paper appeared in the Ph.D. thesis \cite{mythesis} of the author. The author would like to thank his Ph.D. advisor, Paul Terwilliger, for many helpful discussions and comments on this paper.

\bigskip
\noindent Chenwei Ruan \\
Beijing Institute of Mathematical Sciences and Applications \\
No. 544, Hefangkou Village, Huaibei Town, Huairou District \\
Beijing, China 101408 \\
email: \href{mailto:ruanchenwei@bimsa.cn}{\nolinkurl{ruanchenwei@bimsa.cn}}
\end{document}